\documentclass[reqno,twoside,11pt]{amsart}

\usepackage{amsmath,latexsym,amssymb,mathptm, verbatim, enumerate}
    \usepackage{amsfonts,epsfig,textcomp,mathdots,times}%,mathptm}

\usepackage{mathtools}

\usepackage{eucal}
\usepackage{tikz}
\usepackage{dsfont}
\usetikzlibrary{matrix}

\input amssym.def
\input amssym
\input xypic
\input xy

\xyoption{all}
\def\demo{\noindent{\bf Proof. }}
\def\sqr#1#2{{\vcenter{\hrule height.#2pt
        \hbox{\vrule width.#2pt height#1pt \kern#1pt
                \vrule width.#2pt}
        \hrule height.#2pt}}}
\def\square{\mathchoice\sqr64\sqr64\sqr{4}3\sqr{3}3}
\def\QED{\hfill$\square$}

\def\lto{\longrightarrow}

\define\Hom{\operatorname{Hom}}

\def\Homcont{\operatorname{Homcont}}

\def\m{\mathfrak m}

\def\D{\mathcal D}

\newcommand\HH{\mathrm{H}}

\newtheorem{theorem}{Theorem}[section]
\newtheorem{lemma}[theorem]{Lemma}
\newtheorem{corollary}[theorem]{Corollary}
\newtheorem{proposition}[theorem]{Proposition}

\newtheorem{def&dis}[theorem]{Definition and Discussion}
\theoremstyle{definition}
\newtheorem{definition}[theorem]{Definition}
\newtheorem{remark}[theorem]{Remark}

\newtheorem{examples}[theorem]{Examples}

\newtheorem{example}[theorem]{Example}

\begin{document}

\baselineskip=16pt

\title[Simple $\D$-module components of local cohomology modules]{ 
Simple $\D$-module components of local cohomology modules}
\date\today

\author[Robin Hartshorne and  Claudia Polini]
{Robin Hartshorne and Claudia Polini}

\thanks{AMS 2010 {\em Mathematics Subject Classification}.
Primary 13A30, 13H15, 14A10, 13D45; Secondary 13D02,  14E05.}

\thanks{The second author was partially supported by the NSF}

\thanks{Keywords: $\D$-modules, local cohomology, De Rham cohomology}

\address{Department of Mathematics, University of California at Berkeley,
Berkeley, CA 94720} \email{robin@math.berkeley.edu}

\address{Department of Mathematics, 
University of Notre Dame
Notre Dame, IN 46556} \email{cpolini@nd.edu}

 \begin{abstract}   For a projective variety $V \subset \mathbb P^n_k$ over a field of characteristic zero, with homogenous ideal $I$ in $A=k[x_0, \ldots, x_n]$, we consider the local cohomology modules $H^i_I(A)$. These have a structure of holonomic $\D$-module over $A$, and we investigate their filtration by simple $\D$-modules. In case $V$ is nonsingular, we can describe completely the simple $\D$-module components of $H_I^i(A)$ for all $i$, in terms of the Betti numbers of $V$. \end{abstract}
 
\maketitle

%\tableofcontents

\section{Introduction}\label{intro}
\medskip

The local cohomology groups $H^i_Y(X, \mathcal F)$ of an Abelian sheaf $\mathcal F$ on a topological space $X$, with support in a closed subset $Y$, were introduced by 
Grothendieck in 1961 in his Harvard and Paris seminars \cite{GR67}, \cite{GR68}. If $X$ is the spectrum of a ring $A$, and $Y$ is the closed subset associated to an ideal $I$, 
and $\mathcal F$ is  the sheaf of sections of an $A$-module $M$, these groups are denoted $H^i_I(M)$. The latter can be computed algebraically as the derived functors of the functor $\Gamma_I(\bullet)$ that to each $A$-module $M$ associates the submodule of elements with support in $I$, that is, that are annihilated by some power of $I$. In the 50 years since their introduction, these groups and modules have found wide application in algebraic geometry and in commutative algebra. 

Even when the ring $A$ is Noetherian and the module $M$ finitely generated, the local cohomology modules $H^i_I(M)$ are rarely finitely generated. If $A$ is a local ring with maximal ideal $\mathfrak m$, then at least the modules $H^i_{\mathfrak m}(M)$ are {\it cofinite}, meaning that they satisfy the descending chain condition, or equivalently, that ${\rm Hom}_A(k, H^i_{\mathfrak m}(M))$ is a finite dimensional $k$-vector space, where $k =A/\mathfrak m$ is the residue field of $A$. Grothendieck asked whether for any ideal $I$, the modules $H^i_I(M)$ might be {\it $I$-cofinite} in the sense that ${\rm Hom}_A(A/I, H^i_{I}(M))$ is finitely generated. This turns out not to be so in general \cite{ADC}, though an analogous property does hold in the derived category. 

Another finiteness property was discovered more recently by Lyubeznik \cite{L}, who showed that if $A$ is a polynomial ring or a power series over a field $k$ of characteristic zero, then 
for any finitely generated $A$-module $M$ and any ideal $I$ in $A$, the local cohomology modules $H_I^i(M)$ are finitely generated as $\D$-modules, where $\D$ is the (non-commutative) ring of differential operators over $A$. Moreover, they are {\it holonomic} $\D$-modules, in particular they are of finite length as $\D$-modules, and have a finite composition series whose factors are simple $\D$-modules. 

Our basic questions in this paper is, if $V \subset \mathbb P_k^{n}$ is an algebraic variety, with homogeneous ideal $I$ in the polynomial ring $A=k[x_0, \ldots, x_n]$, what are the simple $\D$-module components of the local cohomology modules $H_I^i(A)$, and what can we learn about the geometry of $V$ from them?

To approach this question, we review some old work of Ogus \cite{O}, who found conditions for the vanishing of some of these groups, in terms of the algebraic de Rham cohomology groups of the subvariety $V$. Using a similar technique, we compute the de Rham cohomology groups $H^j_{DR}(M)$ for the $\D$-module $M=H^i_I(A)$, for each $i$ and $j$, in terms of the algebraic de Rham cohomology of $V$. This allows us to recover the result of Ogus on the vanishing and cofiniteness of $H^i_I(A)$ for $i >r$, where $r$ is the codimension. 

Similar results to ours have been obtained in recent years by many authors using various techniques (see Remark~\ref{R4.9}). What is new in our method is that we also obtain information about the $\D$-module structure of the critical cohomology module $H^r_I(A)$. 

To interpret this information, we prove a technical result showing that for a holonomic $\D$-module $M$ over a power series ring $R=k[\![x_1, \ldots,x_n]\!]$, the dimension of the top de Rham cohomology module $H^n_{DR}(M)$ is equal to the largest number of copies of the injective envelope $E$ of $k$ over $R$ whose direct sum $E^m$ can appear as a quotient $\D$-module of $M$. This result, while apparently dual in some way to the elementary statement that $H^0_{DR}(M)$ gives the rank of the largest free $\D$-module $R^m$ that can appear as a submodule of $M$, is not at all easy to prove. 
We use an extension of methods employed by van den Essen \cite{example, vdE, essen} to show that the de Rham cohomology modules $H^j_{DR}(M)$ of a holonomic $\D$-module are all finite dimensional.

Our main application is to show that if $V$ is a nonsingular subvariety of $\mathbb P^{n}_k$ of codimension $r$, then its nontrivial local cohomology module $H_I^r(A)$ has a simple sub-$\D$-module with support along $V$, with quotient a direct sum $E^m$ of copies of $E$, where $m$ is determined by the Betti numbers of $V$ in the sense of algebraic de Rham cohomology. In particular, if $V$ is a nonsingular curve of genus $g$, then $m=2g$, so that for a rational curve $V\subset \mathbb P^{n}_k$, the $\D$-module $H_I^{n-1}(A)$ is simple. The only other result we know giving the simple $\D$-module composition of $H^r_i(A)$ for a projective variety $V$ is the theorem of Raicu \cite{R}, which shows that for $V$ the $d$-uple embedding of another projective space in $\mathbb P^{n}_k$, the corresponding $\D$-module is simple. His result, proved by an entirely different method, is recovered by ours.

One further comment about this paper. Our methods are purely algebraic, working over an algebraically closed field $k$ of characteristic zero. Over the complex numbers $\mathbb C$, there is an extensive theory of analytic $\D$-modules using intersection cohomology and perverse sheaves, and there is a Riemann-Hilbert correspondence comparing the algebraic theory of $\D$-modules with regular singularities to the analytic theory. Readers familiar with those theories will probably see how to obtain  results analogous to ours in the analytic category, and perhaps even recover our result via the Riemann-Hilbert correspondence. Nevertheless our goal has been to present the entire argument algebraically, without reference to the analytic theory.

\bigskip

%%%%%
%%%%SECTION 2
\section{The language of algebraic $\D$-modules }\label{S2}

Let $k$ be a field of characteristic zero, and let $R$ be either the polynomial ring $k[x_1, \ldots,x_n]$ or the formal power series ring $k[\![x_1, \ldots,x_n]\!]$. Let $\D$ be the ring of differential operators $R<\partial_1, \ldots, \partial_n>$, where $\partial_i$ is the partial derivative $\partial/\partial x_i$. This is a non-commutative ring with the relations $\partial_i x_i = x_i \partial_i +1$ for each $i$. An $R$-module $M$, together with a left action of $\D$ on $M$, will be called a $\D$-module. We will use the books of Bj$\ddot{\rm o}$rk \cite{B} and Hotta et al. \cite{HTT} as our basic references. 

One can define the {\it dimension} of a finitely generated $\D$-module. It is an integer between $n$ and $2n$. The $\D$-modules with minimal dimension $n$ are called {\it holonomic} $\D$-modules. They are of finite length as $\D$-modules and therefore have a filtration whose quotients are simple $\D$-modules. (For the polynomial ring case, see \cite[1.5.3]{B}, where these modules are also called modules in the {\it Bernstein class}. For the power series case, see \cite[2.7.13 and the remarks just before 3.3.1]{B}). 

Let $\Omega_{R/k}$ be the module of differentials over $R$, generated by $dx_1, \ldots, dx_n$, and let $\Omega^i_{R/k}$ be its exterior power. If $M$ is a $\D$-module, the actions of $\partial_i$ on $M$ give rise to a complex $M\otimes_R \Omega^{\bullet}$ of $R$-modules and $k$-linear maps, called the {\it de Rham complex} of $M$. Its cohomology groups will be denoted by $H^i_{DR}(M)$. If $M$ is a holonomic $\D$-module, then $H^i_{DR}(M)$ are finite-dimensional $k$-vector spaces. (In the polynomial ring case, the proof is not difficult \cite[1.6.1]{B}. In the power series case, however the question is difficult, and was left as an open problem in Bj\"{o}rk's book. It was proved later by van de Essen \cite[2.2]{essen} as a consequence of his inductive result that if $M$ is holonomic, then for a suitable choice of coordinates, $M/\partial_n M$ will also be holonomic over the power series ring in $n-1$ variables.)

Our interest in $\D$-modules comes from the following  theorem of Lyubeznik. 

\begin{theorem}\label{T2.1} If $M$ is a holonomic $\D$-module over the polynomial ring or the power series ring $R$ as above and if $I$ is an ideal of  $R$, then the local cohomology modules $H_I^i(M)$ have natural structures of holonomic $\D$-modules. 
\end{theorem}
\begin{proof} \cite[2.2]{L}. One first shows that if $M$ is holonomic over $R$ and $f \in R$, then the localized module $M_f$ is also holonomic. (For the polynomial ring case, see for instance \cite[3.4.1]{C}. For the power series case, see \cite[3.4.1]{B}. Then if $f_1, \ldots, f_s$ is a set of generators of $I$, we can compute the local cohomology modules $H_I^i(M)$ from the \v{C}ech complex formed of the localizations of $M$ at products of the $f_i$. Since kernels, images, and quotients of holonomic modules are holonomic, it follows that the $H_I^i(M)$ are holonomic. 
\end{proof}

\begin{examples}\label{E2.2} {\rm 
\begin{enumerate}
\item The ring $R$ itself is a holonomic $\D$-module and  is in fact simple. To see this, we note that any element of $R$ generates $R$ as a $\D$-module. Indeed, just differentiate enough times so that the element becomes a unit, then multiply by the other elements of $R$. 
\item Another important example is $E=H^n_{\mathfrak m}(R)$, where $\mathfrak m=(x_1, \ldots, x_n)$. This is an injective hull of $k$ over $R$, and is also a simple $\D$-module. As in (1) above, any element generates the whole module. Just multiply by enough $x_i$ to arrive at the socle $x_1^{-1}\cdots x_n^{-1}$, then differentiate to get any other monomial in $E$. 
\item The de Rham cohomology of the $\D$-module $R$ is equal to $k$ in degree $0$ and $0$ otherwise. This is a consequence of the algebraic Poincar\'e lemma \cite[II, 7.1]{ADRC}. 
\item The de Rham cohomology of the $\D$-module $E$ is  $k$ in degree $n$, and $0$ otherwise. Look first at the case $n=1$, when $E=A_x/A$ is the $k$-vector space generated by the negative powers of $x$. Consider the map $\varphi: R=k[x] \longrightarrow E$ defined by $\varphi(x^{\ell})=\ell!x^{-\ell-1}$. This is an isomorphism of $k$ vector spaces (where we use the convention that $0!=1$). Note that by construction $\partial \varphi=-\varphi x$. Thus $\varphi$ gives an isomorphism from the complex $R \xrightarrow{x} R$ to the complex $E \xrightarrow{\partial} E$. Taking the tensor product (over $k$) of this isomorphism of complexes over $k[x_i]$, we obtain an isomorphism of the Koszul complex for $R=k[x_1, \ldots, x_n]$ with respect to $x_1, \ldots, x_n$ and the Rham complex for $E$. Hence $H^n_{DR}(E)=k$, and the others are zero. For the case when $R$ is a power series ring, notice that the complex $E \otimes \Omega^{\bullet}$ is the same as in the polynomial ring case.
\item If $M$ is a holonomic $\D$-module whose support, as an $R$-module, is at the maximal ideal $\mathfrak m=(x_1, \ldots, x_n)$, then $M$ is the direct sums of a finite number of copies of $E$. This follows for example from Kashiwara's equivalence \cite[1.6.1, 1.6.4]{HTT}, or one can prove it directly as in \cite[2.4.a]{L}. 
\item For any holonomic $\D$-module $M$, its zeroth de Rham cohomology $H^0_{DR}(M)$ has dimension equal to the rank $t$ of the largest trivial sub-$\D$-module $M_0=R^t$ of $M$. Indeed, if $m\in H^0_{DR}(M)$, then the natural map $R \rightarrow M $ defined by $a \mapsto am$ is an injective $\D$-module homomorphism. One of our main results, Theorem~\ref{main}, is a non-trivial analogous statement about the last de Rham cohomology group $H^n_{DR}(M)$.
\end{enumerate}
}\end{examples}

\bigskip

%%%%%
%%%%SECTION 3

\section{Algebraic de Rham cohomology and homology}\label{S3}

In this section we recall the basic definitions and properties of algebraic de Rham cohomology and homology that we will use in this paper. Our basic references will be Grothendieck \cite{GR66} and Hartshorne \cite{ADRC}. 

Let $Y$ be  a closed subscheme of a scheme $X$ smooth and of finite type of dimension $n$ over an algebraically closed field $k$ of characteristic zero. Let $\Omega^{\bullet}_{X/k}$ be the de Rham complex with $k$-linear maps

\[ \mathcal{O}_X \rightarrow \Omega_X^{1} \xrightarrow{d} \Omega^2_X \rightarrow \cdots \rightarrow \Omega^n_X\, .
\]

We define the algebraic de Rham {\it homology} of $Y$ to be

\[H^{DR}_i(Y)=\mathbb{\HH}^{2n-i}_Y(X, \Omega_X^{\bullet})\, ,\]
namely the local hyper-cohomology with support in $Y$ of the complex $\Omega_X^{\bullet}$ \cite[II.3]{ADRC}. 

We define the algebraic de Rham {\it cohomology } of $Y$ by passing to the formal completion $\mathfrak X$ of $X$ along $Y$ and taking hyper-cohomology
\[H^i_{DR}(Y)=\mathbb{\HH}^{i}(\mathfrak X, \hat{\Omega}_{\mathfrak X}^{\bullet})\, ,\]
of the formal completion of $\Omega_X^{\bullet}$ along $Y$ \cite[II.1]{ADRC}. The main properties of these groups are summarized in the following theorem:

\begin{theorem}\label{T3.1} Let $Y$ be a scheme of finite type over $k$, embeddable in a scheme $X$ smooth over $k$. 
\begin{enumerate} \item The groups $H_i^{DR}(Y)$ and $H^i_{DR}(Y)$ are independent of the embedding of $Y$ in a smooth scheme $X$. 
\item The groups $H_i^{DR}(Y)$ and $H^i_{DR}(Y)$  are finite-dimensional $k$-vector spaces.
\item The groups $H_i^{DR}(Y)$ and $H^i_{DR}(Y)$  are all zero for $i<0$ and $i>2d$, where $d={\rm dim }\, Y$.
\item If $Y$ is proper over $k$, then $H^{DR}_i(Y) \cong H^i_{DR}(Y)^{'}$, where ${}^{'}$ denotes the dual $k$-vector space. 
\item If $Y$ is smooth over $k$, then $H^{DR}_i(Y) \cong H^{2d-i}_{DR}(Y)$.
\item If $Z$ is a closed subset of $Y$, then there is a long exact sequence of homology
\[ \cdots \lto H_i^{DR}(Z) \lto H_i^{DR}(Y) \lto H_i^{DR}(Y-Z) \lto H_{i-1}^{DR}(Z)\lto \cdots
\]
\item If $k=\mathbb C$, then $H^i_{DR}(Y)\cong H^i(Y^{an}, \mathbb C)$, the usual complex cohomology of the associated complex-analytic space $Y^{an}$, and $H_i^{DR}(Y)$ calculates the Borel-Moore homology of $Y^{an}$.

\end{enumerate}
\end{theorem}
\begin{proof} These are all in \cite{ADRC}. Item (1) is II.1.4 and II.3.2; item (2) is II.6.1; item (3) is II.7.2.; item (4) is II.5.1; item (5) is II.3.4; item (6) is II.3.3; and finally, item (7) is IV.1.1 and 1.2.
\end{proof}

\begin{example}\label{E3.2} {\rm \begin{enumerate}
\item If $Y=\mathbb A_k^n$, the affine $n$-space, then by definition $H^i_{DR}(Y)$ is just the de Rham cohomology $H^i_{DR}(R)$ of the polynomial ring $R=k[x_1, \ldots, x_n]$ as a $\D$-module, which is $k$ in degree $0$ and $0$ otherwise (see Example~\ref{E2.2} (3)). Since $Y$ is smooth over $k$ its homology $H_i^{DR}(Y)$ is $k$ in degree $2n$ and $0$ otherwise (see Theorem~\ref{T3.1} (5)).
\item  If $Y=\mathbb P^n_k$, we can show inductively that $H^i_{DR}(Y)=k$ for each $i$ even, $0\le i \le 2n$, and that $H^i_{DR}(Y)=0$ for each $i$ odd. The same is true for homology. Just start with  $\mathbb P^0_k=\mathbb A^0_k$, which has $k$ in degree $0$, and use the long exact sequence of Theorem~\ref{T3.1}(6) with $Z=\mathbb P^{n-1}_k\subset Y=\mathbb P^n_k$ and $Y-Z=\mathbb A_k^n$ to find first the homology, and then use Theorem~\ref{T3.1}(4) to obtain the cohomology of $\mathbb P^n_k$.  \end{enumerate} }
\end{example}

\begin{proposition}\label{P3.3bis}$($ Lichtenbaum theorem for algebraic de Rham cohomology$)$ \begin{enumerate}
\item Let $Y$ be a scheme of dimension $d$ over $k$. Then $H^{2d}_{DR}(Y)\not=0$ if and only if at least one irreducible component of $\, Y$ is proper over $k$. 
\item For $Y$ any scheme, $H_0^{DR}(Y)\not=0$ if and only if $Y$ has at least one connected component that is proper over $k$. 
\end{enumerate} 
\end{proposition}
\begin{proof} This is a straightforward consequence of the result of Theorem~\ref{T3.1}, together with the Mayer-Vietoris sequences and the exact sequences of a birational morphism \cite[4.1, 4.2, 4.4, 4.5]{ADRC}.
\end{proof}

\begin{proposition}\label{P3.3} Let $\mathcal C$ be a nonsingular projective curve of genus $g$ over $k$. Then the dimension of the de Rham cohomology groups are $h^0_{DR}(\mathcal C)=h^2_{DR}(\mathcal C)=1$ and $h^1_{DR}(\mathcal C)=2g$. The homology groups are the same.
\end{proposition}
\begin{proof} In this case, the de Rham complex is just 
 $\mathcal O_{\mathcal C} \xrightarrow{d} \Omega^1_{\mathcal C}\, .
$

There is a spectral sequence
\[ E_1^{pq}=H^q(\mathcal C, \Omega^p) \Longrightarrow E^n=H^n_{DR}(\mathcal C)\, .
\]
The first $d_1$-map is $H^0(\mathcal O_{\mathcal C}) \rightarrow H^0(\Omega_{\mathcal C}^1)$, which is zero, because the only global sections of $\mathcal O_{\mathcal C}$ are constants, and their derivative is zero. Hence $H^0_{DR}(\mathcal C)=k$. By duality, see Theorem~\ref{T3.1}(4) and (5), we see also that $H^2_{DR}(\mathcal C)=k$. Hence the other $d_1$-map $H^1(\mathcal O_{\mathcal C}) \rightarrow H^1(\Omega_{\mathcal C}^1)$ must also be zero  and so $H^1_{DR}(\mathcal C)\cong H^1(\mathcal O_{\mathcal C}) \otimes H^0(\Omega^1_{\mathcal C})$, which has dimension $2g$. The result for homology then follows from Theorem~\ref{T3.1}(5). 
\end{proof}

\begin{proposition}\label{P3.4} Let $\mathcal C$ be an integral projective curve over $k$. For each singular point $P\in \mathcal C$, let $n_P$ be the number of branches of $\mathcal C$ at $P$, that is, the number of points of the normalization $\tilde{\mathcal C}$ of $\mathcal C$ lying over $P$. Then $h^0_{DR}(\mathcal C)=h^2_{DR}(\mathcal C)=1$ and $h^1_{DR}(\mathcal C)=2g+\sum_{P\in \mathcal C} (n_P-1)$, where $g$ is the genus of the normalization  $\tilde{\mathcal C}$. The homology groups are the same. 
\end{proposition}
\begin{proof} We use the exact sequence of homology for a proper birational morphism \cite[II,4.5]{ADRC} applied to the projection $\pi:  \tilde{\mathcal C}\to \mathcal C$. Let $Z$ be the singular locus of $\mathcal C$, and let $Z^{'}$ be its inverse image in  $\tilde{\mathcal C}$. Then we have 
\[ \cdots \lto H_i^{DR}(Z^{'}) \lto H_i^{DR}(Z) \oplus H_i^{DR}( \tilde{\mathcal C} ) \lto H_i^{DR}(\mathcal C) \lto H_{i-1}^{DR}(Z^{'}) \lto \cdots
\]
Since $\tilde{\mathcal C}$ is smooth and projective, its homology (Proposition~\ref{P3.3}) has dimension $h_i^{DR}(\tilde{\mathcal C})=1,2g,1$ for $i=0,1,2$ respectively. The homology of $Z$ and $Z^{'}$ is in degree $0$ only, and is just the number of points in each. Thus 
\[h^{DR}_1(\mathcal C)=h_1^{DR}(\tilde{\mathcal C})+ \#Z^{'}-\#Z\, ,
\]
which gives the result. The same holds for the cohomology of $\mathcal C$ by Theorem~\ref{T3.1}(4).
\end{proof}

  \begin{remark}\label{R3.5}{\rm Of course  Proposition ~\ref{P3.3} and Proposition \ref{P3.4} could have been proved by using the comparison theorem (see Theorem~\ref{T3.1}(7)) and the well-known results about the cohomology of compact Riemann surfaces, but we wished to keep our exposition purely algebraic. } 
\end{remark}

\bigskip

%%%%%
%%%%SECTION 4
\section{Local cohomology of a projective variety}\label{S4}

Now we come to the main subject of our investigation. Let $V$ be a closed subscheme of the projective space $\mathbb P^n_k$ over an algebraically closed field $k$ of characteristic zero. Let $V$ have codimension $r$. Let $A=k[x_0, \ldots,x_n]$ be the homogeneous coordinate ring of $\mathbb P^n_k$ and let $I$ be the homogenous ideal of $V$ in $A$. We propose to investigate the local cohomology modules $H_I^i(A)$. We keep these notations throughout this section. 
\medskip

\begin{proposition}\label{P4.1} Let $V$ be an equidimensional closed subscheme of $\mathbb P^n_k$ of codimension $r$. Let $I$ be the homogenous ideal of $\, V$ in $A=k[x_0, \ldots,x_n]$. 
\begin{enumerate}
\item $H^i_I(A)=0\, $ for $r<i<n+1$.
\item If $V$ is a set-theoretic complete intersection in $\mathbb P^n_k$, then $H^i_I(A)=0$ for $i>r$.
\item If $V$ is a local complete intersection scheme, then for all $r<i\le n+1$, $H^i_I(A)$ has support at the irrelevant maximal ideal $\m=(x_0,\ldots, x_n)$ of $A$. 
\end{enumerate}
\end{proposition}
\begin{proof} Since $V$ has codimension $r$, the ideal $I$ has height $r$ and hence contains a regular sequence of length $r$ for $A$. The first part of assertion (1) now follows from the characterization of $I$-depth in terms of local cohomology and the second part holds because the dimension of the ring is $n+1$. 
For (2), notice that in this case there is an ideal $(f_1, \ldots, f_r)$ generated by $r$ elements having the same radical as $I$, so computing local cohomology using the \v{C}ech complex we obtain $H^i_I(A)=0\, $ for $i>r$.
Assertion (3) is a result of Ogus \cite[4.1, 4.3]{O} proved using a local version of (2). 
\end{proof}

Next we will make use of the $\D$-module structure on the local cohomology modules $H^i_i(A)$.

\begin{proposition}\label{P4.2} Let $X=\mathbb A^{n+1}$ be ${\rm Spec} \, A$, let $I$ be an ideal of $A$, and let $Y={\rm Spec}\, (A/I)$.  Then there is a spectral sequence
\[E_2^{pq}=H^p_{DR}(H_I^q(A))\Longrightarrow H^{DR}_{2n+2-p-q}(Y)
\]
relating the de Rham cohomology of the $\D$-modules $H^i_I(A)$ to the algebraic de Rham homology of the scheme $Y$. 
\end{proposition}
\begin{proof} We compute the algebraic de Rham homology of $Y$ using its embedding in $X$, so that by definition (see Section~\ref{S3})
\[H_i^{DR}(Y)=\mathbb{\HH}^{2n+2-i}_Y(X, \Omega_X^{\bullet})\, .\]
(Note the shift by $2n+2$ since $X$ has dimension $n+1$). Then we use the spectral sequence of local hyper-cohomology of $\Omega_X^{\bullet}$, which is
\[E_1^{pq}=H^q_Y(X, \Omega^p_X)\Longrightarrow E^{p+q}=\mathbb{\HH}^{p+q}_Y(X, \Omega_X^{\bullet})\, .\]
Since $\Omega^p_X$ is a free $\mathcal  O_X$-module, for each $q$ we can write 
\[E_1^{pq}=H^q_I(A)\otimes \Omega^p_X\, .\]
Thus for $q$ fixed, the row $E_1^{pq}$ with differentials $d_1^{pq}$ becomes the de Rham complex of the $\D$-module $H_I^q(A)$, and its homology the $E_2^{pq}$ terms, become the de Rham cohomology $H^p_{DR}(H^q_I(A))$. Thus the spectral sequence of the proposition is the same spectral sequence, but starting with the $E_2$ page. 
\end{proof}

And now, we will see that when $V$ is a local complete intersection, the spectral sequence of Proposition~\ref{P4.2} degenerates. 
\begin{theorem}\label{T4.3} Let $V$ be a local complete intersection in $\mathbb P^n_k$,   equidimensional of  codimension $r$, and let $Y\subset X =\mathbb A^{n+1}$ be defined by the homogenous ideal $I$ of $\, V$. Then
\begin{enumerate}
\item $H^j_{DR}(H^r_I(A))=H^{DR}_{2n+2-j-r}(Y)$ for $0\le j\le n+1$, and 
\item  $H^{n+1}_{DR}(H^i_I(A))=H^{DR}_{n+1-i}(Y)$ for $r\le i\le n+1$.
\end{enumerate}
All other values of  $H^j_{DR}(H^i_I(A))$ are zero, as are all other values of $H^{DR}_{i}(Y)$.
\end{theorem}
\begin{proof} First of all, $H^i_I(A)=0$ for $i<r$ by Proposition~\ref{P4.1}(1). Next, for $i>r$ we know that $H^i_I(A)$ has support at the maximal ideal by Proposition~\ref{P4.1}(3), and hence is isomorphic to a direct sum of copies of $E$ by Example~\ref{E2.2}(5). In that case, $H^j_{DR}(H^i_I(A))=0$ except for $j=n+1$, by Example~\ref{E2.2}(4). Thus the only possible non-zero initial terms of the spectral sequence of Proposition~\ref{P4.2} are for $q=r$ and  $0\le p\le n+1$, or for $p=n+1$ and any $r \le q\le n+1$. This gives a curios $L$-shaped spectral sequence. 
\bigskip

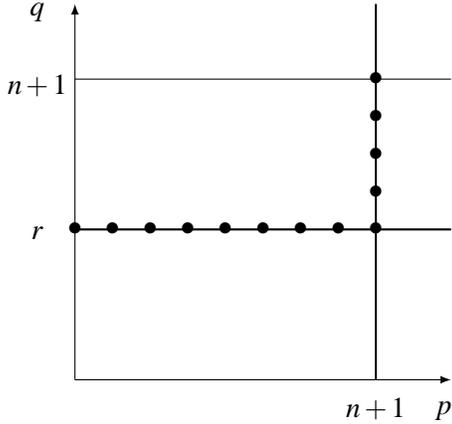
\begin{tikzpicture}         \coordinate (Origin)   at (0,0);
    \coordinate (XAxisMin) at (0,0);
    \coordinate (XAxisMax) at (5,0);
    \coordinate (YAxisMin) at (0,0);
    \coordinate (YAxisMax) at (0,5);
 \draw [thin,-latex] (XAxisMin) -- (XAxisMax);% Draw x axis
    \draw [thin, -latex] (YAxisMin) -- (YAxisMax);% Draw y axis
    \draw [thick] (4,0) -- (4,5);
     \draw [thin] (0,4) -- (5,4);
       \draw [thick] (0,2) -- (5,2);
\node[label=above:$q$] at (-.5,4.5)  {};;
\node[label=above:$n+1$] at (-.5,3.5)  {};
\node[label=above:$r$] at (-.5,1.6)  {};
\node[label=above:$p$] at (4.9,-.8)  {};
\node[label=above:$n+1$] at (4,-.8)  {};
\node[label=above:$\bullet$] at (0,1.65)  {};
\node[label=above:$\bullet$] at (0.5,1.65)  {};
\node[label=above:$\bullet$] at (1,1.65)  {};
\node[label=above:$\bullet$] at (1.5,1.65)  {};
\node[label=above:$\bullet$] at (2,1.65) {};
\node[label=above:$\bullet$] at (2.5,1.65)  {};
\node[label=above:$\bullet$] at (3,1.65) {};
\node[label=above:$\bullet$] at (3.5,1.65)  {};
\node[label=above:$\bullet$] at (4,1.65) {};
\node[label=above:$\bullet$] at (4,2.15) {};
\node[label=above:$\bullet$] at (4,2.65)  {};
\node[label=above:$\bullet$] at (4,3.15)  {};
\node[label=above:$\bullet$] at (4,3.65)  {};
\end{tikzpicture}

\bigskip
There are no non-trivial $d_2$ maps so the spectral sequence degenerates and gives the isomorphism of the theorem.  
\end{proof}

\begin{corollary}\label{C4.4} Let $V, \, I, \, Y$ be as in Theorem~\ref{T4.3}. Then $H^j_{DR}(H^r_I(A))=0$ for $j <r$.
\end{corollary}
\begin{proof} This is because $Y$, being a scheme of dimension $n-r+1$, has no homology in degrees $> 2n-2r+2$, by Theorem~\ref{T3.1}(3). 
\end{proof}

\begin{corollary}\label{C4.5} Let  $V, \, I, \, Y$ be as in Theorem~\ref{T4.3}. For $i>r$ the $\D$-module $H^i_I(A)$ is isomorphic to $E^{m_i}$ where $m_i={\rm dim}_k\,  H^{DR}_{n+1-i}(Y)$.  \end{corollary}
\begin{proof} We have seen in the proof of Theorem~\ref{T4.3} that for $i>r$  we have $H^i_I(A)\cong E^{m_i}$ for some $m_i$. Since $H^j_{DR}(E)=0$ for $j\not=n+1$ and $k$ for $j=n+1$, this number $m_i$ is the dimension of $H^{n+1}_{DR}(H^i_I(A))$, which is equal to the dimension of $H^{DR}_{n+1-i}(Y)$, according to  the theorem. 
\end{proof}

Our next task is to relate the algebraic de Rham homology of $Y$ to that of $V$. 

\begin{proposition}\label{P4.6} With the hypotheses of Theorem~\ref{T4.3}, assume furthermore that $V$ is connected,  of dimension $d \ge 1$. Then 
\begin{enumerate}
\item $H_0^{DR}(Y)=H_1^{DR}(Y)=0$,
\item $H_2^{DR}(Y)\cong H_1^{DR}(V)$,
\item $H_{2d+2}^{DR}(Y)\cong H_{2d}^{DR}(V)$, and 
\item for $3\le i\le 2d+1$ the homology of $Y$ is determined by the exact sequence
\[0=H_{2d+1}^{DR}(V)\xrightarrow{h}H_{2d-1}^{DR}(V)\to H_{2d+1}^{DR}(Y) \to H_{2d}^{DR}(V)\to \ldots
%\xrightarrow{h}H_{2d-2}(V)\to 
%\ldots \to H_{3}(V)\xrightarrow{h}H_{1}(V)\to 
\to H_{3}^{DR}(Y) \to H_2^{DR}(V)\xrightarrow{h}H_{0}^{DR}(V)\to 0
\]
where $h$ denotes cap-product with the hyperplane class.
\end{enumerate}
\end{proposition}
\begin{proof} We use a method of proof similar to \cite[II,3.2]{ADRC} but with homology instead of cohomology. The first step is to compare the homology of $Y$ to that of $Y-P$, where $P$ corresponds to $\m=(x_0,\ldots, x_n)$ in $A$. Since $P$ only has homology in degree zero, the exact sequence of Theorem~\ref{T3.1}(6) gives an exact sequence
\[0\to H_1^{DR}(Y)\to H_1^{DR}(Y-P)\to H_0^{DR}(P)\to H_0^{DR}(Y)\to H_0^{DR}(Y-P)\to 0\, ,
\]
and isomorphisms
\[H_i^{DR}(Y)\cong H_i^{DR}(Y-P) \mbox{ \quad for \ } i\ge 2
\]
Next we note that $Y-P$ is isomorphic to the geometric vector bundle $\mathbb V(\mathcal O_V(-1))$ minus its zero section, so we can apply the Thom-Gysin sequence \cite[II,7.9.3]{ADRC} to obtain a long exact sequence 
\[\ldots \to  H_i^{DR}(V)\xrightarrow{h} H_{i-2}^{DR}(V)\to H_i^{DR}(Y-P)\to H_{i-1}^{DR}(V) \to \ldots
\]
where $h$ is the cap-product in homology \cite[II, 7.4]{ADRC}. From the last terms of this sequence it follows that $H_0^{DR}(Y-P)=0$ and $H_1^{DR}(Y-P)\cong H_0^{DR}(V)=k$, since $V$ is connected (see also Proposition~\ref{P3.3bis}). Since $H_0^{DR}(P)=k$ and $H_0^{DR}(Y)=0$ by Proposition~\ref{P3.3bis}, the earlier sequence now implies that $H_1^{DR}(Y)=0$. Since $V$ has dimension $\ge 1$, it follows that the cap product $h:H_2^{DR}(V)\to H_0^{DR}(V)$ is surjective, and so $H_2^{DR}(Y)\cong H_2^{DR}(Y-P)\cong H_1^{DR}(V)$. (To see that $h$ is surjective, note that it is dual to the cup-product $H^0_{DR}(V)\to H^2_{DR}(V)$, and the image of the generator of $H^0_{DR}(V)$ is the hyperplane class in $H^2_{DR}(V)$, which, having self-intersection equal to the degree of $V$ must be non-zero.) Now using $H_i^{DR}(Y)\cong H_i^{DR}(Y-P)$ for $i \ge 3$ gives the desired assertions (3) and (4). 
\end{proof}

\begin{corollary}\label{C4.7}$[$Ogus, \cite[4.4]{O}$]$ Let $V$ be a local complete intersection in $\mathbb P^n_k$,  equidimensional of codimension $r$, connected, of dimension $d \ge 1$, with homogenous ideal $I$ in $A$. Then the groups $H^i_I(A)$ are zero for all $i>r$ if and only if the restriction maps
\[H^j_{DR}(\mathbb P^n) \longrightarrow H^j_{DR}(V)
\]
are isomorphisms for all $j<n-r$. 
\end{corollary}
\begin{proof} By Theorem~\ref{T4.3} it follows that the $H^i_I(A)=0$ for all $i>r$ is equivalent to $H^{DR}_j(Y)=0$ for all $j<n+1-r$. This in turn, by Proposition~\ref{P4.6}, is equivalent to saying that the cap-product $h:H_j^{DR}(V)\to H_{j-2}^{DR}(V)$ is an isomorphism for all $j<n-r$ and surjective for $j=n-r$. By duality (see Theorem~\ref{T3.1}(4)) this is equivalent to saying that the cup-product $H^{j-2}_{DR}(V)\to H^j_{DR}(V)$ is an isomorphism for all $j<n-r$ and injective for $j=n-r$. Beginning with $H^{-1}_{DR}(V)=0$ and $H^0_{DR}(V)=k$, and using the fact that the cohomology of projective space is $0$ in odd degrees and $k$ in even degrees generated by the hyperplane class $h \in H^2_{DR}(\mathbb P^n)$ (see Example~\ref{E3.2}(2)), our calculation is equivalent to saying that the restriction map $H^j_{DR}(\mathbb P^n) \to H^j_{DR}(V)$ is an isomorphism for all $j<n-r$.
\end{proof}

As an illustration of these results, we gather together our conclusions for a nonsingular variety.
 
 \begin{theorem}\label{T4.8} Let $V$ be a nonsingular irreducible variety in $\mathbb P^n_k$   of  codimension $r$ and dimension $d=n-r \ge 1$, with homogenous ideal $I$ in $A$. Then writing $b_i={\rm dim}\, H^{DR}_i(V)$ for the Betti numbers of $\, V$, we have
 \begin{enumerate}
 \item $H^i_I(A)=E^{m_i}$ with $m_i=b_{n-i}-b_{n-i-2}\ $ for $r<i<n$ and $H^i_I(A)=0$ for $i \ge n$.
 %\item for $i>r$, $H^i_I(A)=E^{m_i}$ with $m_i=b_{n-i}-b_{n-i-2}\ $ for $i<n$, and $m_n=m_{n+1}=0$. 
 %\item $H^j_{DR}(H^r_I(A))=0$ for $j<r$, and has dimension $b_{\ell}-b_{\ell-2}$ for $j=r+\ell<n$, and $b_d-b_{d-2}$ for $j=n,n+1$. 
 \item $H^j_{DR}(H^r_I(A))=0\, $ for $j<r$, and has dimension $b_{n+d-j}-b_{n+d-j+2}\, $ for $r\le j \le n$, and $b_d-b_{d-2}\ $ for $j=n+1$. 
 \end{enumerate}
\end{theorem}
\begin{proof} Since $V$ is nonsingular of dimension $d$, the hard Lefschetz theorem tell us that capping with the hyperplane class $h$ gives a map 
\[h:H_i^{DR}(V)\lto H_{i-2}^{DR}(V)
\]
that is surjective for $i\le d+1$ and injective for $i\ge d+1$. Therefore by 
%examining the exact sequence of 
Proposition~\ref{P4.6}, we find that 
$$
h_i^{DR}(Y)=\begin{cases}
  0 \hspace{2cm}\text{for} \quad i=0,1 \\
b_1 \hspace{1.85cm} \mbox{for} \quad i=2 \\
  b_{i-1}-b_{i-3} \hspace{.5cm} \mbox{for} \quad 3\le i\le d+1 \\
   b_{i-2}-b_{i} \hspace{.85cm} \mbox{for} \quad d+1< i\le 2d \\
     b_{2d-1}\hspace{1.35cm} \mbox{for} \quad i=2d+1 \\
          b_{2d}\hspace{1.35cm}\quad \mbox{for} \quad i=2d+2 \, .
\end{cases}
$$
Substituting these values in the statement of Theorem~\ref{T4.3} gives the desired assertions. 
\end{proof}

  \begin{remark}\label{R4.9}{\rm  Many of the results of this section concerning the case $i>r$ are not new. What is new are the results concerning the $\D$-module structure of the nontrivial case $H^r_I(A)$, especially Theorem~\ref{T4.3}(1) and Theorem~\ref{T4.8}(2). 
  
  The understanding of the relationship  between cofiniteness and vanishing of the local cohomology modules $H^i_I(A)$ for $i>r$ and the algebraic de Rham cohomology of the projective variety $V$ goes back to Ogus \cite{O}. This connection is acknowledged in the last paragraph of Lyubeznik's paper \cite{L}, just after he has defined some new numerical invariants of a local ring, commonly called {\it Lyubeznik numbers}. The study of these numbers  has led to several results analogous to ours. 
  
  Garci\`{a}-L\'{o}pez and Sabbah \cite{Lopez} give a result similar to our Corollary~\ref{C4.5} for an isolated singularity of a complex analytic space, in terms of local topological invariants. 
  
Blickle and Bondu \cite{BB} give a similar result for a point $P$ in a complex analytic space $Y$ under the condition that $Y - P $ is an intersection homology manifold. This condition is probably equivalent to   Ogus's condition on the $DR$-depth \cite[4.1]{O}, which is in fact {\it equivalent} to the local cohomology modules $H^i_I(X)$ being cofinite for $i>r$. 

Lyubeznik, Singh, and Walther \cite[3.1]{Singh} give another analogue of our Corollary~\ref{C4.5} over $\mathbb C$, taking as {\it hypothesis} that these local cohomology groups have support at $\m$, and computing the $m_i$ in terms of the complex singular cohomology of $\mathbb C^n \setminus Y$. 

Switala \cite{S} in a recent paper about Lyubeznik numbers recovers independently our Theorem~\ref{T4.8}(1) for the vertex of the cone over a nonsingular projective variety. His argument is similar to ours, but uses cohomology instead of homology. 
  } 
\end{remark}

\bigskip

%%%%%
%%%%SECTION 5
\section{$\D$-modules over the power series ring}\label{Main}

%\begin{data}\label{data1}  
%\end{data} 

\begin{theorem}\label{main} Let $A$ be the power series ring $k[\![x_1, \ldots,x_n]\!]$, let $E$ be an injective hull of $k$ over $A$, and let $M$ be a holonomic $\D$-module. If $m=\dim_k H^n_{DR}(M)$, then there is a surjective homomorphism of $\D$-modules
\[M\lto E^m \lto 0\, .
\] 
\end{theorem}
\demo  Recall that $H^n_{DR}(M)$ is a finite dimensional $k$-vector space (see Section \ref{S2}). Observe that since $H^n_{DR}(M)$ is the homology of the last term of the de Rham complex, it is simply $M/(\partial_1, \ldots, \partial_n)M$. Choose a linear map from $M/(\partial_1, \ldots, \partial_n)M$ to $k$ and compose it with the canonical epimorphism from $M$ to $M/(\partial_1, \ldots, \partial_n)M$ to obtain a map $\psi$ from $M$ to $k$. Because of Proposition \ref{P3} below,  $\psi$ is a continuous map in the sense of Definition \ref{cont}. Therefore, by Proposition \ref{P1}, the map $\psi$ corresponds to an $A$-linear map $\varphi$ from $M$ to $E$. The correspondence in Proposition \ref{P1} depends on the choice of a $k$-linear projection $\pi$ of $E$ to its socle $k$. We now choose $\pi $ to be the projection of $E$ to $E/(\partial_1, \ldots, \partial_n)E$, which is isomorphic to the socle of $E$. Then we have a diagram
%Because of the way the map $\varphi$ was constructed in Proposition \ref{P1}, 
$$\xymatrix{
\ar@{->}[d]M\ar@{->}[r]^{\varphi}&E\ar@{->}[d]^{\pi}\\
M/(\partial_1, \ldots, \partial_n)M\ar@{->}[r]&k=E/(\partial_1, \ldots, \partial_n)E}$$ 

\bigskip
\noindent 
which shows that $\varphi$ maps the kernel of $\psi$ to the kernel of $\pi$. Hence 
%we have $(\partial_1, \ldots, \partial_n)M \subset \ker \, \psi$ is mapped by $\varphi$ to $ \ker \, \pi =(\partial_1, \ldots, \partial_n)E$, 
\[\varphi((\partial_1, \ldots, \partial_n)M)\subset (\partial_1, \ldots, \partial_n)E\, .\]

Now according to Proposition \ref{P2} the map $\varphi$ is not only $A$-linear, but  is also a map of $\D$-modules. Further, observe that since $E$ is a simple $\D$-module and $\varphi$ is not zero,  $\varphi$ is surjective. 
Applying the same reasoning to a basis for the space of linear maps from $M/(\partial_1, \ldots, \partial_n)M$ to $k$ we obtain $m$ maps from $M$ to $E$ and therefore a single surjective map from $M$ to $E^m$. This completes our proof, subject to Propositions \ref{P1}, \ref{P2}, \ref{P3} below. 
% once we show that $\psi$ is a continuous map in the sense of Definition \ref{cont}. This last statement is a consequence of Proposition \ref{P3}. Indeed, for every finitely generated $A$-submodule $N$ of $M$ there exists an integer $s$ such that $\m^s N \subset  (\partial_1, \ldots, \partial_n)M$, according to Proposition \ref{P3}, thus $\psi (\m^sN) =0$. 
\QED

\begin{corollary}\label{MC} Let $A$ be the power series ring $k[\![x_1, \ldots,x_n]\!]$, let $E$ be an injective hull of $k$ over $A$, and let $M$ be a holonomic $\D$-module. Then  \[\dim_k H^n_{DR}(M)=\dim_k \Hom_{\D}(M,E)\, .\]\end{corollary}
\demo By the theorem we have  \[\dim_k H^n_{DR}(M)\le \dim_k \Hom_{\D}(M,E)\, .\] Conversely, if $\dim_k \Hom_{\D}(M,E)=s$, then there is a surjective map from $M$ to $E^s$. Apply the de Rham cohomology functor. Then 
\[H^n_{DR}(M) \lto H^n_{DR}(E^s)\lto 0\]
Since $H^n_{DR}(E)=k$, we have \[\dim_k H^n_{DR}(M)\ge \dim_k H^n_{DR}(E^s)=s\, .\]
\QED

Now we are ready to prove the three propositions that are the main ingredients in the proof of Theorem \ref{main}.
%The proof of the main theorem is build on the following three propositions.

\begin{definition}\label{cont} Let $(A, \m, k)$ be a local ring that contains its residue field $k$. Let $M$ be an $A$-module. A $k$-linear homomorphism $\psi$ of $M$ to $k$ is called {\it continuous} if for every finitely generated submodule $N$ of $M$ there is an integer $s$ such that $\psi(\m^sN) =0$. We denote the $A$-module of continuous linear homomorphisms by $\Homcont_k(M,k)$. 
\end{definition}

The following proposition appears in \cite{GR68, SAr16} but we give the proof for convenience. 

\begin{proposition}\label{P1}\cite[IV, Remarque 5.5]{GR68} Let $(A, \m, k)$ be a local ring that contains its residue field $k$, and $E$ an injective hull of $k$. For any $A$-module $M$, the Matlis dual $\Hom_A(M,E)$ is isomorphic as a $k$-vector space to the module  $\Homcont_k(M,k)$. 
\end{proposition}
\demo  Choose a $k$-linear projection $\pi$ of $E$ to its socle $k$. Then for any $\varphi \in \Hom_A(M,E)$, composing with $\pi$, we obtain a $k$-linear homomorphism $\psi$ from $M$ to $k$. Let us show that $\psi$ is continuous. For any finitely generated submodule $N$ of $M$ the image $\varphi(N)$ is a finitely generated submodule of $E$ and therefore is an $A$-module of finite length. Thus there exists an integer $s$ such that $\m^s\varphi(N)=0$. It follows that $\varphi(\m^s N)=0$ and thus $\psi(\m^s N)=0$. Hence $\psi$ is continuous in the sense of Definition \ref{cont}. 
We have thus constructed a $k$-linear map $\Lambda$ from $\Hom_A(M,E)$ to $\Homcont_k(M,k)$. Now we show that $\Lambda$ is an isomorphism. If $M=k$, it is obvious. For an $A$-module of finite length, the statement follows by induction on the length of the module and short exact sequences
\[ 0\lto M' \lto M \lto M''\lto 0
\]
and the fact that $\Hom_A(M,E)$ and  $\Homcont_k(M,k)$ are contravariant exact functors. Next, if $M$ is a finitely generated $A$-module, then every homomorphisms of either $\Hom_A(M,E)$ or   $\Homcont_k(M,k)$ factors through $M/\m^{\ell}M$ for some $\ell$, hence we have 
\[ \Hom_A(M,E)= \varinjlim \Hom_A(M/\m^{\ell} M,E)\simeq \varinjlim \Homcont_k(M/\m^{\ell} M,k)=\Homcont_k(M,k)\, .\] For $M$ an arbitrary $A$-module, think of $M$ as the direct limit of its finitely generated submodules, $M= \varinjlim M_{\ell}$. Thus we have
\begin{eqnarray*}
\Hom_A(M,E)&=&  \Hom_A( \varinjlim M_{\ell},E)= \varprojlim \Hom_A(M_{\ell},E)\\ &\simeq& \varprojlim \Homcont_k(M_{\ell},k)
= \Homcont_k( \varinjlim M_{\ell},k)=\Homcont_k(M,k)\, .
\end{eqnarray*}

 \QED

\begin{proposition}\label{P2} Let $A$ be the power series ring $k[\![x_1, \ldots,x_n]\!]$, let $E$ be an injective hull of $k$ over $A$,  let $M$ be a $\D$-module, and let $\varphi: M \lto E$ be an $A$-linear map such that 
\[\varphi(\partial M)\subset \partial E\, , \]
where $\partial=(\partial_1, \ldots, \partial_n)$. Then $\varphi $ is also $\D$-linear. 
\end{proposition}
\demo We must show that $\partial_i\varphi(m)=\varphi(\partial_i m)$ for all $i$ and for all $m \in M$. Observe that both sides are elements of $\partial E$. The left hand side because it is a $\partial_i$ of something, the right hand side because of our hypothesis that $\varphi(\partial M)\subset \partial E$. Next we note that the map $\partial E \lto E^n$, sending  $e \in \partial E$ to $(x_1e, \ldots, x_n e) \in E^n$, is injective. Therefore it sufficient to prove that for every $j$, 
\begin{equation*}\tag{$\ast$}
x_j\partial_i\varphi(m)=x_j\varphi(\partial_i m) \quad \forall m \in M\, .
\end{equation*}
We claim that statement ($\ast$)  is equivalent to showing 
\[\partial_i\varphi(x_j m)=\varphi(\partial_i (x_j m)) \quad \forall m \in M\, .
\]
If $i \not=j$ the claim is clear because $x_j$ and $\partial_i$ commute and $\varphi$ is $A$-linear. If $i=j$, then we use the equation $x_i\partial_i=\partial_i x_i -1$ in the ring of differential 
operators. Indeed, notice that the left hand side of ($\ast$) is
\[x_i\partial_i(\varphi(m))=\partial_i\varphi(x_i m)-\varphi(m)\]
while the right hand side of ($\ast$) is
\[x_i\varphi(\partial_i m)=\varphi(x_i\partial_i m)=\varphi(\partial_i(x_i m)-m)=\varphi(\partial_i(x_i m)-\varphi(m)\, .\]
Now  after canceling $\varphi(m)$ we obtain the desired claim. We have thus replaced the original problem for $m\in M$  by the same problem for $x_jm$. Repeating the same procedure it is sufficient to show 
\[\partial_i\varphi(\alpha m)=\varphi(\partial_i (\alpha m)) \quad \forall m \in M
\]
for all  monomials $\alpha \in A$ of any high degree we like. To conclude notice that both sides are zero for degree of $\alpha$ sufficiently large.  Indeed, the left hand side is clearly zero since $\varphi(\alpha m)=\alpha \varphi(m)$ and $\varphi(m)\in E$. The right hand side is zero because using the product rule we have 
\[\varphi(\partial_i (\alpha m))=\varphi(\partial_i(\alpha) m)+ \varphi(\alpha\partial_i m)=\partial_i(\alpha)\varphi(m) + \alpha\varphi(\partial_i m)\]
and both $\partial_i(\alpha)$ and $\alpha$ have sufficiently high degree. 
\QED

\begin{remark}{\rm The statement and the proof of Proposition~\ref{P2} also hold over a polynomial ring or its localization at the maximal ideal.}
\end{remark}

\begin{theorem}$[$\cite[3.3.19]{B},\cite[Prop. 1]{example}$]$\label{bjork}  Let $A$ be the power series ring $k[\![x_1, \ldots,x_n]\!]$ and let $M$ be a holonomic $\D$-module over $A$. Then there exists a nonzero element $g\in A$ such that $M[g^{-1}]$ is a holonomic $\D$-module that is finitely generated as an $A[g^{-1}]$-module. Furthermore, after a linear change of variables, we may assume that $g(x_1, 0, \ldots, 0)\not=0$, and in that case we can take $g$ to be a Weierstra\ss \  polynomial
\[x_1^r+a_1x_1^{r-1}+ \ldots+a_r\, ,\]
with $a_i\in k[\![x_2, \ldots,x_n]\!]$. In this situation we say that $M$ is {\em $x_1$-regular}. 
\end{theorem}

\begin{theorem}\cite[Thm I]{vdE} \label{Essen} Let $A$ be the power series ring $k[\![x_1, \ldots,x_n]\!]$ and let $M$ be a holonomic $\D$-module over $A$ that is $x_1$-regular. Then $M/\partial_1(M)$ is a holonomic $\D$-module over the ring $k[\![x_2, \ldots,x_n]\!]$. 
\end{theorem}

\begin{remark}{\rm Van den Essen showed by an example \cite{example} that for an arbitrary holonomic $\D$-module $M$, the quotient $M/\partial_1(M)$ need not to be a holonomic $\D$-module over the ring  $k[\![x_2, \ldots,x_n]\!]$. However, with the extra condition that $M$ is $x_1$-regular, this holds. }
\end{remark}

\begin{lemma}\label{L1} Let $A$ be the power series ring $k[\![x_1, \ldots,x_n]\!]$, let $M$ be a holonomic $\D$-module that is $x_1$-regular, and let $e$ be any element of $M$. Then there exists a differential operator of the form
\[P=a_0+a_1 \partial_1+\ldots+ a_r\partial_1^r \qquad a_i\in A
\] where $a_r$ has a pure power of $x_1$, and such that \[P(b)\cdot e\subset \partial_1(M) \quad \forall b\in A \, .\] 
\end{lemma}
\demo While not given exactly in this form, our statement and proof are based on a careful reading of \cite{vdE}. According to Theorem \ref{bjork} there is an elements $g$ such that $M[g^{-1}]$ is  finitely generated as an $A[g^{-1}]$-module,  and furthermore $g$ can be taken to be  a Weierstra\ss \  polynomial in $x_1$. Let $x=x_1$ and $\partial=\partial_1$. Let $e$ be an element of $M$. Then $e, \partial e, \partial^2 e, \ldots, \partial^i e, \ldots$ are linearly dependent over $A[g^{-1}]$. Therefore, there exists an integer $r$ and elements $c_i \in A[g^{-1}]$ such that
\[\partial^r e= \sum_{0}^{r-1}c_i \partial^i e\, .\]
Clearing denominators we can write
\[g^s \partial^r e= \sum_{0}^{r-1}d_i \partial^i e\, ,\]
where $d_i \in A$. So we can consider the differential operator
\[Q=\sum_{0}^{r} d_i \partial^i \, ,\]
using the $d_i$ above for $0\le i \le r-1$ and $d_r=-g^s$. By construction $Q(e)=0$ in $M$ and $d_r$ is a Weierstra\ss \  polynomial in $x$.
There exist $a_i \in A$ such that the differential operator $Q$ can be written as 
\[Q=a_0-\partial a_1+\ldots+(-1)^r \partial^r a_r\, .\]
Let $P$ be the differential operator
\[P=a_0+a_1 \partial+\ldots+ a_r\partial^r \, .\]
Then for every $b\in A$ we claim there is an equality of differential operators
\begin{equation}\label{eq3}bQ=P(b)+\partial R\, ,\end{equation}
where $R$ is another differential operator and $P(b)$ means $P$ acting on $b\in A$. Then, if we apply the two operators defined in (\ref{eq3}) to $e$ we have,
\[0=bQ(e)=P(b) e+\partial R (e)\, .\]
This implies that for all $b\in A$, $P(b) e \subset \partial (M)$, thus establishing the desired conclusion. 

To complete the proof we need to prove the equality of differential operators in (\ref{eq3}). By linearity it suffices to show the claim for $Q=(-1)^i\partial^i a_i$ and $P=a_i \partial^i$. We need to show that $bQ-P(b) = \partial R_i$ for some operator $R_i$. To prove the claim we  show by induction on $i$ a stronger claim, more precisely,  that for all $f, g \in A$
\[(-1)^i  f\partial^i g=g\partial^i(f)+\partial R_i\,  ,\]
for some operator $R_i$. If $i=0$, we can take $R_0=0$ because $f g =gf$. If $i=1$ we have 
\[f\partial +\partial(f)=\partial f\, .\] Since $\partial(f)\in A$ we have $g\partial(f)= \partial(f) g$. Thus take $R_1=  -\partial fg$. Let $i \ge 2$ and assume that the statement holds for $i-1$. We have
\[(-1)^if \partial^i g= (-1)^i f \partial (\partial^{i-1} g)= (-1)^i [\partial f -\partial (f)] (\partial^{i-1} g)= (-1)^i \partial f \partial^{i-1} g +(-1)^{i-1}\partial(f)\partial^{i-1} g\]
By induction hypothesis we have
\[(-1)^{i-1}\partial(f)\partial^{i-1} g=g\partial^{i-1}(\partial(f))+\partial R_{i-1}\,\]
thus substituting in the previous equation we obtain 
\[(-1)^if \partial^i g= g\partial^{i-1}(\partial(f))+\partial R_{i-1}+(-1)^i \partial f \partial^{i-1} g= g\partial^i (f)+ \partial R_i\]
where $R_i= R_{i-1}+(-1)^i f \partial^{i-1} g$, which proves the claim.

\QED

\medskip 
The following lemma is our key technical result. 
\begin{proposition}\label{L2} Let $A$ be the power series ring $k[\![x_1, \ldots,x_n]\!]$, let $B=k[\![x_2, \ldots,x_n]\!]$, let $x=x_1$, and let $\partial=\partial_1$. Let $P$  be a differential operator of the form
\[P=\sum_{i=0}^r a_i \partial^i \qquad  a_i\in A\,,\] where $a_r$ has a pure power of $x$. Then there exist integers $s$ and $\ell_0$ such that every  $f \in A$ can be written in the form
\[f= \sum_{i=0}^{s-1} e_ix^i + \sum_{\ell \ge \ell_0} b_\ell P(x^\ell)\]
where $e_i, b_\ell \in B$. Furthermore, if $f\in A x^m$ for some $m \ge s$ then 
\[e_i \in \m_B^{\rho(m)} \quad \forall \ 0\le i \le s-1\, ,
\]
where $\rho$ is a function that tends to infinity with $m$. 
\end{proposition}
\demo We first show that there exist integers $\ell_0$ and $t$ such that for all $\ell \ge \ell_0$ 
\[P(x^\ell)=\sum_{i=0}^{\infty}c_{\ell,i}x^i\]
where $c_{\ell,i} \in B$ satisfy the following conditions
\begin{itemize}
\item[(i)] $ c_{\ell,i}=0 \quad \ \quad \mbox{ \ \ \qquad  for }  i <\ell-r $
\item[(ii)]$ c_{\ell,i}\in \m_B \quad \quad \mbox{ \quad \quad for } \ell-r \le i <\ell-r+t $
\item[(iii)] $c_{\ell,i} \mbox{ is a unit in } B \quad  \mbox{ for } i=\ell-r+t $.
\end{itemize}
(There is no restriction on $c_{\ell,i}$ for $i > \ell-r+t$. ) Condition (i) is clear because $x^\ell$ has degree $\ell$ and $\partial^r$ is the highest differential in $P$.
To prove (ii) and (iii), write for $\ell \ge r$
\[P(x^\ell)=a_0x^\ell+\ell a_1  x^{\ell-1}+2{\ell \choose 2}a_2x^{\ell-2}+\ldots+ r! {\ell \choose r} a_rx^{\ell-r}\, .\]
Let $t$ be the least power of $x$ whose coefficient is a unit in B among all the power series
\[\{a_0 x^{r}, a_1x^{r-1}, \ldots, a_r\}\, .\]
Notice that $t$ exists because by hypothesis $a_r$ has a pure power of $x$.
Let $\lambda_i \in k$ be the constant term of the coefficient of $x^t$ in $a_ix^{r-i}$. By construction, at least one of the $\lambda_i$ is non zero. The constant term of the coefficient of $x^{\ell+t-r}$ in $P(x^\ell)$ is
\[g(\ell)=\lambda_0+\ell\lambda_1+\ldots + r! {\ell \choose r} \lambda_r\, .\]
Since $g(\ell)$ is a non zero polynomial in $\ell$, it has at most finitely many zeros. Choose $\ell_0$ such that $g(\ell)\not=0$ for all $\ell\ge \ell_0$. Now for any $\ell \ge \ell_0$ write $P(x^\ell)=\sum_{i=0}^{\infty}c_{\ell,i}x^i$, where $c_{\ell,i} \in B$. Then by construction the $c_{\ell,i}$ will satisfy the condition (ii) and (iii), namely, the first one that is a unit is $c_{\ell,\ell-r+t}$. 

To continue set $s=\ell_0-r+t$. Then we show that every $f\in A$ can be written in the form
\[f= \sum_{i=0}^{s-1} e_ix^i + \sum_{\ell \ge \ell_0} b_\ell P(x^\ell)\]
where $e_i, b_\ell \in B$. We claim, by induction on $k$, that 
\begin{equation}\label{eq2}
f\equiv f_k=\sum_{i=0}^{s-1} e_{i,k}x^i + \sum_{\ell \ge \ell_0} b_{\ell,k} P(x^\ell) \qquad \mod \m_B^k
\end{equation}
for suitable $e_{i,k}$ and $b_{\ell,k}$ in $B$ and furthermore $e_{i,k}\equiv e_{i,k+1}$ and $b_{\ell,k}\equiv b_{\ell,k+1}$ modulo $\m_B^k$.  First observe that for all $\ell \ge \ell_0$ 
\begin{equation}\label{eq1}
P(x^\ell)\equiv \alpha_{\ell} x^{\ell-r+t} \mod \m_B \, ,
\end{equation}
where $\alpha_{\ell}$ is a unit in $A$. Here we use that $c_{\ell,i} \in \m_B$ for all $0\le i < \ell-r+t$. If the original power series $f$ is
\[f= \sum_{i=0}^{\infty} \beta_i x^i\]
with $\beta_i\in B$, then we define $f_1$ using  $e_{i,1}=\beta_i$ for $0\le i \le s-1$ and $b_{\ell,1}=\beta_{\ell-r+t}(\alpha_\ell)^{-1}$ for $\ell \ge \ell_0$. Thus $f \equiv f_1 \mod \m_B$ and the claim follows for $k=1$. For the $k+1$ step, we consider $f-f_k$. Notice that the coefficients $\gamma_{i,k}$ of $f-f_k$ as a power series in $x$ are all in $\m_B^k$. Therefore when we use $P(x^\ell)$ to adjust the $\ell-r+t$ coefficient of $f-f_k$, we have by (\ref{eq1})
%\[\gamma_{s+n-n_0,k} c_{n,i} \in \m_B^{k+1} \quad \forall \ 0 \le i < s+n-n_0\, \] 
%thus modulo $\m_B^{k+1}$ we again have 
\[\gamma_{\ell-r+t,k} P(x^\ell)\equiv \gamma_{\ell-r+t}\alpha_{\ell} x^{\ell-r+t} \mod \m_B^{k+1} \, .\]
Define $f_{k+1}=f_k +\sum_{\ell \ge \ell_0}\gamma_{\ell-r+t,k} (\alpha_\ell)^{-1}P(x^\ell)$. Then $f \equiv f_{k+1} \mod \m_B^{k+1}$.  Writing $f_{k+1}$ in the form (\ref{eq2}) we obtain the coefficients  $e_{i,k+1}$  and $b_{\ell,k+1}$ and observe that, by construction,  they are congruent to the coefficients $e_{i,k}$  and $b_{\ell,k}$ modulo $\m_B^k$. Now the desired assertion follows by passing to the limit: namely $e_i=\lim e_{i,k}$ and $b_\ell=\lim b_{\ell,k}$ as $k$ goes to infinity. 

To explain the 'Furthermore' statement, suppose that $f\in A x^m$. Then $f$, as a power series in $x$, begins in degree $\ge m$. We claim that $f_k$ starts in degree $\ge m-kt$ and the coefficients of $x^{j}$ for $m-pt\le j< m-(p-1)t$ are in $\m_B^{p}$. Recall that $P(x^\ell)$ as a power series in $x$ begins in degree $\ge \ell-r$ but its first unit coefficient is in degree $\ell-r+t$. We prove the claim by induction on $k$. For $k=1$ recall that $b_{\ell,1}=\beta_{\ell-r+t}(\alpha_\ell)^{-1}$ and the first $\beta_i$ that can be different from zero is $\beta_m$, hence the first $P(x^\ell)$ we are using is for $\ell=m+r-t$ and that can only start in degree $\ell-r=(m+r-t)-r=m-t$. Furthermore, the coefficients of $x^j$ for $m-t\le j < m$ of $P(x^{m+r-t})$ as power series in $x$  are in $\m_B$ by condition (ii) on the $c_{\ell,i}$ above, or, by the fact that $f_1\equiv f \ \mod \m_B$.  By the construction used to build $f_{k}$ from $f_{k-1}$ we see that $f_{k}$ satisfies our claim. Now the coefficients of $f$ are obtained by taking the limits of the coefficients of the $f_k$. Since $s$ is fixed and $m$ can be taken as large as we like, we have that $e_i$ for $0\le i \le s-1$ are contained in $\m_B^{\rho(m)}$ where $\rho(m)$ is a function that tends to infinity to $m$, approximately equal to $(m-s)/t$. \QED

\begin{remark}{\rm A result similar to this was proved by van den Essen \cite{vdE} but without the ``Furthermore" statement, which is crucial  to our proof.}
\end{remark}

\begin{proposition}\label{P3} Let $A$ be the power series ring $k[\![x_1, \ldots,x_n]\!]$, let $M$ be a holonomic $\D$-module, and let $N$ be a  finitely generated submodule  of $M$. Then there exists an integer $r$ such that \[\m^r N \subset (\partial_1, \ldots, \partial_n) M\, .\] 
\end{proposition}
\demo We may assume that $N$ is generated by one element $e$. By Theorem \ref{bjork} we can make a change of variables so that $M$ is $x_1$-regular. Let $x=x_1$ and $\partial=\partial_1$.  By Lemma \ref{L1} there exists a differential operator of the form given in the lemma such that $P(a) e \subset \partial M$ for all $a\in A$. We apply Proposition \ref{L2} to this differential operator. By  Proposition \ref{L2} there exists an integer $s$ such that for all $f\in  A x^m$ with $m\ge s$ we have
\begin{equation}\label{eq4}f \cdot e \subset \m_B^{\rho(m)} E + \partial M\, ,\end{equation}
where $B=k[\![x_2, \ldots,x_n]\!]$ and $E=B( e, xe, \ldots, x^{s-1}e)$, because $P(x^\ell)e \subset \partial M$ for all $\ell$ by Lemma \ref{L1}.

We show the statement by induction on the number of variables  $n$. If $n=1$ then $B=k$ and thus $\m_B=0$. Therefore, by (\ref{eq4}) for all $f\in  A x^m$ with $m\ge s$ we have 
\[f \cdot e \subset  \partial M\,   ,\]
hence we obtain
\[ \m^s N \subset \partial M\, ,  \]
which is the desired assertion for $n=1$. 
If $n \ge 2$, we apply  Theorem \ref{Essen}, which  says $\overline{M}=M/\partial M$ is a holonomic $\D$-module over $B$. Let $\overline{E}$ be the image of $E$ in $\overline{M}$, then $\overline{E}$ is a finitely generated submodule of $\overline{M}$. By the induction hypothesis there exists an integer $t$  such that 
\[\m_B^{t} \overline{E}\subset (\partial_2, \ldots, \partial_n) \overline{M}\, .\]
This implies that 
\begin{equation}\label{eq5}\m_B^{t} E\subset (\partial_1, \ldots, \partial_n) M\, .\end{equation}
By Proposition \ref{L2} we can take $m_0$ large enough so that $m_0\ge s$ and $\rho(m_0)\ge t$. Take $r$ to be $r=m_0+\rho(m_0)$. We claim that  \[\m_A^r N=\m_A^r e \subset (\partial_1, \ldots, \partial_n) M\, .\] 
Write any monomial $\alpha \in  \m_A^r$ as $x^i \gamma$. Notice that either $i\ge  m_0$ or $\gamma \in \m_B^j$ with $j \ge  \rho(m_0)\ge t$. In the first case $\alpha \in x^{m_0} A$, hence by 
(\ref{eq4}) (notice (\ref{eq4}) applies because $m_0\ge s$)
\[\alpha e \subset \m_B^{\rho(m_0)} E + \partial M  \subset (\partial_1, \ldots, \partial_n) M\]
where the last inclusion hold by (\ref{eq5})  and because we chose $m_0$ in such a way $\rho(m_0) \ge t$. 
In the second case, the claim follows directly by (\ref{eq5}).

\QED

\bigskip

%%%%%
%%%%SECTION 6

\section{Applications.}\label{App}

\medskip

We want to apply the results of Section~\ref{Main} to study the local cohomology module $H_I^r(A)$, where $V$ is a variety in $\mathbb P^n_k$ of codimension $r$, with homogenous ideal $I \subset A=k[x_0, \ldots, x_n]$. Unfortunately the result of Theorem~\ref{main} is not true for $\D$-modules over a polynomial ring  (see Example~\ref{E6.1}). So we need to pass to the completion.

If $M$ is a $\D$-module over the polynomial ring $A=k[x_0, \ldots, x_n]$, we can consider the completion $M\otimes_A \hat{A}$ where $\hat{A}=k[\![x_0, \ldots,x_n]\!]$. It has a natural structure of  $\D$-module over $\hat{A}$. Furthermore, there is a natural map on the de Rham cohomology groups: $ H^i_{DR}(M)\lto H^i_{DR}(\hat{M})$. Unfortunately, even for a holonomic $\D$-module over $A$, the map on de Rham cohomology may fail to be an isomorphism (see Example~\ref{E6.1}). It would be nice to have general conditions under which these maps are isomorphisms. For instance, if $M$ is also a graded $A$-module, and the $\partial_i$ act as graded $k$-linear maps of degree $-1$, then are the completion maps $\varphi^i$ isomorphisms? We do not know, so we will settle for a more limited criterion.

\begin{example}\label{E6.1} {\rm Let $A=k[x]$ and let $M$ be a free $A$-module of rank one. We denote its generator by $e\in M$, so that the elements of $M$ are written $ae$ for $a \in A$.  To give $M$ a structure of $\D$-module we can take $\partial e$ to be anything we like. So for example, let $\partial e=x^2 e$. Then for any power of $x$, $\partial(x^ne)=x^n\partial e+nx^{n-1}e=(x^{n+2}+nx^{n-1})e$, and we can extend to all of $M$ by linearity. Now it is clear that the map $\partial: M \lto M$ is injective, so that $H^0_{DR}(M)=0$. On the other hand, the image of $\partial$ is a  $k$-vector subspace of codimension $2$, so $H^1_{DR}(M)=M/\partial M$ has dimension $2$. This shows that Theorem~\ref{main} is false for $M$, since any nonzero $\D$-module homomorphism of $M$ to $E$ would have to be surjective, which is impossible since $M$ is finitely generated as an $A$-module.  

Now let consider $\hat{M}=M\otimes_A \hat{A}$. We claim that there is another generator of $\hat{M}$, call it $ue$, with $u\in \hat{A}$ a unit, for which $\partial(ue)=0$. To find $u$, we want $\partial(ue)=0$, hence $u\partial e+\partial u \cdot e=0 $. We need to solve the differential equation $\partial u=-x^2u$. Just take $u=e^{-\frac{x^3}{3}}$, which is a unit in $\hat{A}$. Thus $\hat{M}$ is isomorphic to the standard $\D$-module structure on $\hat{A}$ with $H^0_{DR}(\hat{M})=1$ and $H^1_{DR}(\hat{M})=0$. So we see that the passage to the completion does not preserve de Rham cohomology.  }
\end{example}

\begin{theorem}\label{T6.2} Let $A=k[\![x_0, \ldots,x_n]\!]$, let  $I\subset A$ be a homogenous ideal, and let $M$ be any of the local cohomology modules $M=H^i_I(A)$ with its $\D$-module structure. Then the completion maps 
\[H^j_{DR}(M)\lto H^j_{DR}(\hat{M})
\] where $\hat{M}=M\otimes_A \hat{A}$, are isomorphisms for all $j$. Furthermore $\hat{M}$ is also holonomic over $\hat{A}$. 
\end{theorem}
\begin{proof} The last statement follows from Theorem~\ref{T2.1} and the fact that
\[\hat{M}=H^i_I(A)\otimes_A \hat{A}=H^i_{\hat{I}}(\hat{A})\, .
\]
We will {\it pull ourselves up by our bootstraps} using the earlier results of Section \ref{S4}. First of all, since the local cohomology $H^i_I(A)$ can be computed from the \v{C}ech complex of localizations of $A$ at product of the $f_i$, where $\{f_1, \ldots, f_s\}$ is a set of generators of $I$, and a short exact sequence of modules gives a long exact sequence of de Rham cohomology, we reduce to the case where $I=(f)$ is generated by a single homogeneous polynomial $f\in A$. In this case there is only one non-zero local cohomology, namely, $M=H^1_f(A)=A_f/A$. Let $V$ be the corresponding hypersurface in $\mathbb P^n_k$, and let $Y\subset \mathbb A^{n+1}_k$ be the affine cone over $V$, namely, the affine subscheme of ${\rm Spec}\, A$ defined by $f$. Then the hypotheses of Theorem~\ref{T4.3} are satisfied and hence by that theorem we have that for each $j$ 
\[H^j_{DR}(M)=H^{DR}_{2n+1-j}(Y)
\]
the algebraic de Rham homology of $Y$. 

The same proof as for Proposition~\ref{P4.2} and Theorem~\ref{T4.3}, carried out over the formal power series ring $\hat{A}=k[\![x_0, \ldots,x_n]\!]$, will show that
\[H^j_{DR}(\hat{M})=H^{DR}_{2n+1-j}(Y')
\]
where $Y'\subset {\rm Spec}\, \hat{A}$ is the subscheme defined by the same polynomial $f\in \hat{A}$. Here we use the local theory of algebraic de Rham cohomology and homology (\cite[III]{ADRC}). Therefore we just need to show that the natural maps $H^{DR}_i(Y)\lto H^{DR}_i(Y')$ are isomorphisms for all $i$. This is a question purely in the theory of algebraic de Rham homology, which we prove next. 
\end{proof}

\begin{proposition}\label{P6.3} $($Strong excision for homology$)$ Let $I$ be a homogeneous ideal in $A=k[x_0,\ldots,x_n]$ and let $Y$ be the affine scheme in ${\rm Spec}\, A=\mathbb A_k^{n+1}$ defined by $I$. Let $\hat{A}=k[\![x_0, \ldots,x_n]\!]$, let $\hat{I}=I\hat{A}$, and let $Y'\subset {\rm Spec}\, \hat{A}$ be defined by $\hat{I}$. Then there are natural isomorphisms of de Rham homology
\[H^{DR}_i(Y)\lto H^{DR}_i(Y')
\]
for each $i$. 
\end{proposition}
\begin{proof} Let $V\subset \mathbb P^n_k$ be the projective scheme defined by $I$. Then according to \cite[III, 3.2]{ADRC} there is an exact sequence
\[\ldots \to H_{i+1}(Y')\to H_i(V) \xrightarrow{} H_{i-2}(V)\to H_i(Y')\to \ldots
\]
We have established in Proposition~\ref{P4.6}(4) the same sequence with $H_i(Y)$ in place of $H_i(Y')$. Note that this does not depend on the special hypotheses of Theorem~\ref{T4.3} and Proposition~\ref{P4.6}. Since there are compatible maps between these sequences, we conclude that $H_i(Y)\xrightarrow{\sim} H_i(Y')$ for all $i$. 
\end{proof}

\begin{theorem}\label{T6.4}$($Main Theorem$)$ Let $V$ be a nonsingular variety of codimension $r$ in the projective space $\mathbb P^n_k$. Let $I$ be the homogenous ideal of $\, V$ in $A=k[x_0,\ldots,x_n]$. Then the local cohomology module $M=H^r_I(A)$ has a simple sub-$\D$-module $M_0$ with support on $Y$, the cone over $V$, and the quotient $M/M_0$ is  a direct sum of $b_d-b_{d-2}$ copies of $E$, the injective hull of $k$ over $A$, where $d={\rm dim}\, V$, and $b_i$ are the Betti numbers $b_i={\rm dim}\, H_i^{DR}(V)$. 
\end{theorem}
\begin{proof} By Corollary~\ref{C4.5} , $H^{n+1}_{DR}(H^r_I(A))=H^{DR}_{n+1-r}(Y)$. On the other hand, since $V$ is nonsingular and ${\rm dim}\, V=d=n-r$, by Theorem~\ref{T4.8} the dimension of this homology group is $b_d-b_{d-2}$.

Next, letting $M=H^r_I(A)$, we apply Theorem~\ref{T6.2} to see that $H^j_{DR}(M)=H^j_{DR}(\hat{M})$ for each $j$. Therefore, by Theorem~\ref{main} and Corollary~\ref{MC}, we have a surjective map $\hat{M}\lto E^m$ with $m=b_d-b_{d-2}$, and ${\rm dim}\, \Hom_{\D}(\hat{M},E)=m$. Composing with the natural map $M \lto \hat{M}$ we obtain a map $M \lto E^m$, which must be surjective since $H^{n+1}_{DR}(M)\cong H^{n+1}_{DR}(\hat{M})$.

On the other hand, since $Y$ has only one singular point at $P$, it follows from the general theory \cite[Section 3.4]{HTT}, that if we take the simple $\D_Y$-module $\mathcal O_Y$ on the smooth part $Y-P$ of $Y$, then $M$ contains a simple $\D$-module, the {\it minimal extension} of $\mathcal O_Y$ to $X$ in the sense of \cite[3.4.2]{HTT} and that the quotient $M/M_0$ has support at $P$. That quotient must be a sum of copies of $E$ (see Example~\ref{E2.2} (5)), and therefore is equal to the quotient $E^m$ found above. 
\end{proof}

\begin{corollary}\label{C6.5} If $\, V=\mathcal C$ is a nonsingular curve of genus $g$ in $\mathbb P^n_k$, then there is  just one nonzero local cohomology group $M=H^{n-1}_I(A)$. It has a simple sub-$\D$-module $M_0\subset M$ supported on the cone over $\mathcal C$, and the quotient $M/M_0$ is isomorphic to $E^{2g}$. In particular, if $\mathcal C$ is a nonsingular rational curve, then $M=H^{n-1}_I(A)$ is a simple $\D$-module. 
\end{corollary}
\begin{proof} Indeed, ${\rm dim}\, H_i(V)=1,2g,1$ for $i=0,1,2$, respectively, hence, $m=b_1-b_{-1}=2g$. 
\end{proof}

\begin{corollary}\label{C6.6} If $\, V$ is any embedding of a projective space $\mathbb P^d$ in another projective space $\mathbb P^n$, then the local cohomology group $M=H_I^{n-d}(A)$ is a simple $\D$-module. 
\end{corollary}
\begin{proof} Indeed, the homology of $V$ has dimension $1$ in even degree and dimension zero in odd degrees, so in any case $b_d-b_{d-2}=0$. 
\end{proof}

\begin{remark}\label{R6.7}{\rm In the special case of the Veronese embedding of $\mathbb P^d_k$ in $\mathbb P^n_k$, this was proved via an entirely different method, using representation theory, by Raicu \cite{R}. 

}
\end{remark}

\begin{remark}\label{R6.8}{\rm For a singular projective variety $V\subset \mathbb P^n_k$ of codimension $r$, think of a stratification of $V$ by locally closed nonsingular subvarieties. Then we expect $M=H^r_I(A)$ to have one simple sub-$\D$-module corresponding to the smooth part of $V$, and a succession of contributions coming from the strata of the singular locus, and finally a quotient that is equal to ${\rm dim}\, H^{DR}_{d+1}(Y)$ copies of $E$ as before. 

In particular, if $V=\mathcal C$ is an integral curve, then we expect one component for the smooth part of $\mathcal C$, then at each singular point $P$, $n_P-1$ copies of the injective hull of the line cone over $P$ (in the notation of Proposition~\ref{P3.4}), and then $h_1(\mathcal C)$ copies of $E$.

This may be clear to readers familiar with the Riemann-Hilbert correspondence and perverse sheaves, but we have not worked out details of the proof in our algebraic formulation. 

}
\end{remark}

\bigskip

\bigskip

\bigskip\noindent{\bf Acknolwedgment.} Part of this paper was written at the 
 Centre International de Rencontres Math\'emati\-ques (CIRM) in Luminy, France, while the authors participated in a Recherches en Bin\^{o}me.  We are very appreciative of the hospitality offered by the Soci\'et\'e Math\'ematique de France. In addition, we would like to thank Alexander Beilinson, Claudiu Raicu, and Uli Walther  for helpful discussions on the subject.


\begin{thebibliography}{99}


\bibitem{BB}{M. Blickle and R. Bondu, {\em Local cohomology multiplicities in terms of \'{e}tale cohomology}, Ann. Inst. Fourier {\bf 55}, (2005), 2239--2256.}%

\bibitem{B}{J.-E. Bj\"ork, {\em Rings of differential operators}, North-Holland Mathematical Library \textbf{21}, North-Holland Publishing, Amsterdam-New York, 1979.}%

\bibitem{C}{S. C. Coutinho, {\em A primer of algebraic $\D$-modules}, Lond. Math. Soc. Student Text {\bf 33} (1995).}%
\bibitem{H2}{R. Hartshorne, {\em Cohomological dimension of algebraic varieties}, Ann. of Math. {\bf 88} (1968), 403--450.}%
		
\bibitem{ADC}{R. Hartshorne, {\em Affine duality and cofiniteness}, Invent. Math. {\bf 9} (1970), 145--164.}%

\bibitem{ADRC}{R. Hartshorne, {\em On the De Rham cohomology of algebraic varieties}, Inst. Hautes Etudes Sci. Publ. Math. {\bf 45} (1975), 5--99. }%



%\bibitem{He}{M. Hellus, {\em On the associated primes of Matlis duals of top local cohomology modules}, Comm. Algebra {\bf 33} (2005), 3997--4009. }
	
\bibitem{HTT}{ R. Hotta, K. Takeuchi, and T. Tanisaki, {\em $\D$-modules, perverse sheaves, and representation theory}, 
Progress in Mathematics, vol. 236, Birkhauser Boston, Inc., Boston, MA, 2008.} %

%\bibitem{EGA IV}{A. Grothendieck, \'{E}l\'ements de g\'eom\'etrie alg\'ebrique. {IV}. \'{E}tude locale des sch\'emas et des morphismes de sch\'emas {IV}, Inst. Hautes \'Etudes Sci. Publ. Math. {\bf 32} (1967), 5-361.}

\bibitem{GR66}{A. Grothendieck, {\em On the de Rham cohomology of algebraic varieties},  Inst. Hautes Etudes Sci. Publ. Math. {\bf 29} (1966), 95--103.}%

\bibitem{GR67}{A. Grothendieck, {\em Local Cohomology}, Springer Lecture Notes {\bf 41} (1967). }%

\bibitem{GR68}{A. Grothendieck, {\em Cohomologie locale des faisceaux coh\'erents et th\'eor\`emes de {L}efschetz locaux et globaux {$(SGA$} {$2)$}}, North-Holland Publishing Co., Amsterdam (1968).}%

%\bibitem{HS}{C. Huneke and  G. Lyubeznik, On the vanishing of local cohomology modules, Invent. Math.  {\bf102} (1990), 73--93. }


\bibitem{Lopez}{R. Garc\'{i}a L\'{o}pez and C. Sabbah, {\em Topological computation of local cohomology multiplicities}, Collect. Math. {\bf 49} (1998), 317--324.}%

%\bibitem{L89}{G. Lyubeznik, {\em A survey of problems and results on the number of defining equations}, Commutative Algebra, MSRI Publ. 1989, pp. 375--390.}

\bibitem{L}{G. Lyubeznik, {\em Finiteness properties of local cohomology modules (an application of {$\D$}-modules to commutative algebra)}, Invent. Math. {\bf 113} (1993), 41--55. }%

\bibitem{Singh}{G. Lyubeznik, A. Singh, and U. Walther,  {\em Local cohomology modules supported at determinantal ideals}, preprint 2014. }%


\bibitem{O}{A. Ogus, {\em Local cohomological dimension of algebraic varieties},  Ann. of Math. {\bf 98} (1973), 327--365.}%

\bibitem{R}{C. Raicu, {\em Characters of equivariant $\D$-modules on Veronese cones}, Trans. Amer. Math. Soc. , to appear.}%

\bibitem{S}{N. Switala, {\em  Lyubeznik numbers for nonsingular projective varieties},  Bull. Lond. Math. Soc.  {\bf 47} (2015), 1--6. }%


\bibitem{SAr15}{N. Switala, {\em Van den Essen's theorem on the de Rham cohomology of a holonomic $\D$-module over a formal power series ring}, preprint (2015), arXiv:1603.09743.}%

%\bibitem{ST}{N. Switala, {\em Some invariants of nonsingular projective varieties and complete local rings}, Thesis (Ph.D.)�University of Minnesota (2015).}


\bibitem{SAr16}{N. Switala, {\em On the de Rham homology and cohomology of a complete local ring in equicharacteristic zero}, preprint (2016), arXiv:1505.01788.}%



%\bibitem{essthesis}{A. van den Essen, {\em Fuchsian modules}, thesis, Katholieke universiteit Nijmegen (1979).}

%\bibitem{kernel}{A. van den Essen, {\em Le noyau de l'op\'{e}rateur $\frac{\partial}{\partial x_n}$ agissant sur un $\mathcal{D}_n$-module}, C.R. Acad. Sci. Paris \textbf{288} (1979), 687-690.}

\bibitem{example}{A. van den Essen, {\em Un $\mathcal{D}$-module holonome tel que le conoyau de l'op\'{e}rateur $\frac{\partial}{\partial x_n}$ soit non-holonome}, C.R. Acad. Sci. Paris \textbf{295} (1982), 455-457.}%

\bibitem{vdE}{A. van den Essen, {\em Le conoyau de l'op\'{e}rateur $\frac{\partial}{\partial x_n}$ agissant sur un $\mathcal{D}_n$-module holonome}, C.R. Acad. Sci. Paris \textbf{296} (1983), 903-906.}%

%\bibitem{several}{A. van den Essen, {\em The kernel and cokernel of a differential operator in several variables},  Indag. Math. \textbf{45} (1) (1983), 67-76.}

%\bibitem{several2}{A. van den Essen, {\em The kernel and cokernel of a differential operator in several variables II},  Indag. Math. \textbf{45} (4) (1983), 403-406.}

\bibitem{essen} {A. van den Essen, {\em The cokernel of the operator $\frac{\partial}{\partial x_n}$ acting on a $\mathcal{D}_n$-module II}, Compos. Math. \textbf{56} (2) (1985), 259-269.}%
\end{thebibliography}
\end{document}